\documentclass[<opyions>]{elsarticle}
\usepackage{amsfonts}
\usepackage{tipa}
\usepackage{mathrsfs}
\usepackage{amssymb}
\usepackage{latexsym,bm}
\usepackage{amsthm}
\usepackage{yhmath}
\usepackage{booktabs}
\usepackage{multirow}
\usepackage[center]{caption}
\usepackage{graphics}
\usepackage{subfig}
\usepackage{color}
\usepackage{algorithmic}
\usepackage{algorithm}
\usepackage{geometry}
\usepackage{textcomp}
\newtheorem{remark}{Remark}
\newtheorem{definition}{Definition}
\newtheorem{lemma}{Lemma}
\newtheorem{corollary}{Corollary}
\newtheorem{theorem}{Theorem}
\newtheorem{example}{Example}[section]

\begin{document}
\begin{frontmatter}
\title{High-order numerical methods for the Riesz space fractional advection-dispersion equations}
\author[els]{L.B.~Feng}
\ead{fenglibo2012@126.com}
\author[rvt,rvh]{P.~Zhuang}
\ead{zxy1104@xmu.edu.cn}
\author[els]{F.~Liu\corref{cor1}}
\ead{f.liu@qut.edu.au.} \cortext[cor1]{Corresponding
author:~f.liu@qut.edu.au. (F.~Liu).}
\author[els]{I.~Turner}
\ead{i.turner@qut.edu.au.}
\author[elq]{J.~Li}
\ead{lijingnew@126.com}

\address[els]{School of Mathematical Sciences, Queensland University of Technology,  GPO Box 2434, Brisbane, Qld. 4001, Australia}
\address[rvt]{School of Mathematical Sciences, Xiamen University, Xiamen 361005, China}
\address[rvh]{Fujian Provincial Key Laboratory of Mathematical Modeling and
High-Performance Scientific Computation, Xiamen University, Xiamen 361005, China}
\address[elq]{School of Mathematics and Computing Science, Changsha University of Science and Technology, Changsha 410114, China}

\begin{abstract}
In this paper, we propose high-order numerical methods for the Riesz space fractional advection-dispersion equations (RSFADE) on a {f}inite domain. The RSFADE is obtained from the standard advection-dispersion equation by replacing the first-order and second-order space derivative with the Riesz fractional derivatives of order $\alpha\in(0,1)$ and $\beta\in(1,2]$, respectively. Firstly, we utilize the weighted and shifted Gr\"unwald difference operators to approximate the Riesz fractional derivative and present the {f}inite difference method for the RSFADE. Specifically, we discuss the Crank-Nicolson scheme and solve it in matrix form. Secondly, we prove that the scheme is unconditionally stable and convergent with the accuracy of $\mathcal {O}(\tau^2+h^2)$. Thirdly, we use the Richardson extrapolation method (REM) to improve the convergence order which can be $\mathcal {O}(\tau^4+h^4)$. Finally, some numerical examples are given to show the effectiveness of the numerical method, and the results are excellent with the theoretical analysis.
\end{abstract}
\end{frontmatter}
\section{Introduction}\label{sec1}
There have been increasing interests in the description of the physical and chemical processes by means of equations involving fractional derivatives over the last decades. And, fractional derivatives have been successfully applied into many sciences, such as physics \cite{b1}, biology \cite{b2}, chemistry \cite{b3}, hydrology \cite{b4,b5,b6,b7}, and even finance \cite{b8}. In groundwater hydrology the fractional advection-dispersion equation (FADE) is
utilized to model the transport of passive tracers carried by fluid flow in a porous medium \cite{b6,b9}.

Considerable numerical methods for solving the FADE have been proposed. Kilbas et al. \cite{Kilbas06} introduced the theory and applications of fractional differential equations. Meerschaert and Tadjeran \cite{b10} developed practical numerical methods to solve the one-dimensional space FADE with variable coefficients on a finite domain. Liu et al. \cite{b6} transformed the space fractional Fokker-Planck equation into a system of ordinary differential equations (method of lines), which was then solved using backward differentiation formulas. Momani and Odibat \cite{b9} developed two reliable algorithms, the Adomian decomposition method and variational iteration method, to construct numerical solutions of the space-time FADE in the form of a rapidly convergent series with easily computable components. Zhuang et al. \cite{b11} discussed a variable-order fractional advection-diffusion equation with a nonlinear source term on a finite domain. Liu et al. \cite{b12} proposed an approximation of the L\'evy-Feller advection-dispersion process by employing a random walk and finite difference methods. In addition, other finite difference methods \cite{b13}, finite element method \cite{b23}, finite volume method \cite{b24}, homotopy perturbation method \cite{b25} and spectral method \cite{b26,b27} are also employed to approximate the FADE.

In this paper, we consider the following RFADE:
\begin{equation}\label{eq1}
\frac{\partial u(x,t)}{\partial t}=K_{\alpha}\frac{\partial ^\alpha u(x,t)}{\partial \left| x \right|^\alpha}+K_{\beta}\frac{\partial^\beta u(x,t)}{\partial \left| x \right|^\beta},\quad 0<x<L,\quad 0<t\leq T,
\end{equation}
subject to the initial condition:
\begin{equation}\label{eq2}
u(x,0)=\psi(x),\quad\quad 0\le x\le L,
\end{equation}
and the zero Dirichlet boundary conditions:
\begin{equation}\label{eq3}
u(0,t)=0,\quad u(L,t)=0,\quad 0\le t\le T,
\end{equation}
where $0<\alpha<1$, $1<\beta\leq2$, $K_{\alpha}\geq0$ and $K_{\beta}>0$ represent the average fluid velocity and the
dispersion coefficient. The Riesz space fractional operators $\frac{\partial^\alpha u}{\partial \left| x \right|^\alpha }$ and $\frac{\partial^\beta u}{\partial \left| x \right|^\beta }$ on a {f}inite domain $[0,L]$ are
defined respectively as
\begin{align*}
\frac{\partial^\alpha u(x,t)}{\partial \left|x\right|^\alpha }=-c_\alpha\left[{_0D_x^\alpha u(x,t)}+{_xD_L^\alpha u(x,t)}\right],\quad
\frac{\partial^\beta u(x,t)}{\partial\left|x\right|^\beta}=-c_\beta\left[{_0D_x^\beta u(x,t)}+{_xD_L^\beta u(x,t)}\right],
\end{align*}
where $c_\alpha=\frac{1}{2\cos \frac{\pi\alpha}{2}}$, $c_\beta=\frac{1}{2\cos \frac{\pi\beta}{2}}$,
and
\begin{align*}
{_0D_x^\alpha u(x,t)}&=\frac{1}{\Gamma (1-\alpha )}\frac{\partial}{\partial x}\int_{0}^{x}{(x-\xi })^{-\alpha}u(\xi ,t)d\xi,\\
{_xD_L^\alpha u(x,t)}&=\frac{-1}{\Gamma (1-\alpha )}\frac{\partial}{\partial x}\int_{x}^{L}{(\xi -x})^{-\alpha }u(\xi ,t)d\xi,\\
{_0D_x^\beta u(x,t)}&=\frac{1}{\Gamma (2-\beta )}\frac{\partial^2}{\partial x^2}\int_{0}^{x}{(x-\xi })^{1-\beta}u(\xi
,t)d\xi,\\
{_xD_L^\beta u(x,t)}&=\frac{1}{\Gamma (2-\beta )}\frac{\partial^2}{\partial x^2}\int_{x}^{L}{(\xi -x})^{1-\beta }u(\xi ,t)d\xi ,
\end{align*}
where $\Gamma(\cdot)$ represents the Euler gamma function.

The fractional kinetic equation (\ref{eq1}) possesses a physical meaning (see \cite{b21,b22} for further details). Physical considerations of a fractional advection-dispersion transport model restrict $0 <\alpha<1$, $1 < \beta \leq2$, and we assume $K_\alpha\geq 0$ and $K_\beta>0$ so that the flow is from left to right. In the case of $\alpha=1$ and $\beta=2$, Eq.(\ref{eq1}) reduces to the classical advection-dispersion equation (ADE). In
this paper, we only consider the fractional cases: when $K_\alpha=0$, Eq.(\ref{eq1}) reduces to the Riesz fractional diffusion equation (RFDE) \cite{b21} and when $K_\alpha\neq0$, the Riesz fractional advection-dispersion equation (RFADE) is obtained \cite{b22}.

For the RFADE (\ref{eq1}), Anh and Leonenko \cite{b18} presented a spectral representation of the mean-square solution without the non-homogeneous part and for some range of values of the parameters. Later, Shen et al. \cite{b19} derived the fundamental solution of Eq.(\ref{eq1}) and discussed the numerical approximation of Eq.(\ref{eq1}) using finite difference method with first convergence order. Another method based on the spectral approach and the weak solution formulation was given in Leonenko and Phillips \cite{b20}. In addition, Zhang et al. \cite{b14} use the Galerkin finite element method to approximate the RFADE. Besides, Yang et al. \cite{b28} applied the $L1/L2$-approximation method, the standard/shifted Gr\"unwald method, and the matrix transform method (MTM) to solve the RFADE. And, Ding et al. \cite{b29} also consider the numerical solution of the RFADE by using improved matrix transform method and the (2, 2) Pade approximation. Most of the numerical methods proposed by these
authors are low order or lack stability analysis.

In this paper, based on the weighted and shifted Gr\"unwald difference (WSGD) operators to approximate the Riesz space fractional derivative, we obtain the second order approximation of the RFADE. Furthermore, We propose the {f}inite difference method for the RFADE and obtain the Crank-Nicolson scheme. Moreover, we prove that the Crank-Nicolson scheme is unconditionally stable and convergent with the accuracy of $\mathcal{O}(\tau^2+h^2)$ and improve the convergence order to $\mathcal{O}(\tau^4+h^4)$ by applying the Richardson extrapolation method.

The outline of the paper is as follows. In Section \ref{sec2}, the WSGD operators and some lemmas are given. In Section \ref{sec3}, we first present the {f}inite difference method for the RFADE, and then derive the Crank-Nicolson scheme. We proceed with the proof of the stability and convergence of the Crank-Nicolson scheme in Section \ref{sec4}. Besides, we further improve the convergence order by applying the Richardson extrapolation method. In order to verify the effectiveness of our theoretical analysis, some numerical examples are carried out and the results are compared with the exact solution in Section \ref{sec5}. Finally, the conclusions are drawn.

\section{The approximation for the Riemann-Liouville fractional derivative}\label{sec2}
First, in the interval $[a,b]$, we take the mesh points $x_i=a+ih$, $i=0, 1,\cdots,m$, and $t_n=n\tau$, $n=0, 1,\cdots,N$, where $h=(b-a)/M$, $\tau= T/N$, i.e., $h$ and $\tau$ are the uniform spatial step size and temporal step size. Now, we give the definition of the Riemann-Liouville fractional derivative.

\begin{definition}
[\cite{b15}]\label{def1} The $\gamma~(n-1<\gamma<n)$ order left and right Riemann-Liouville fractional derivatives of the function $v(x)$ on $[a,b]$, are given by
\begin{itemize}
\item[$\bullet$] left Riemann-Liouville fractional derivative:
\begin{align*}
{_aD_x^\gamma v(x)}=\frac{1}{\Gamma (n-\gamma )}\frac{{{\mathrm{d}}^{n}}}{\mathrm{d} {{x}^{n}}}\int_{a}^{x}{(x-\xi)}^{n-\gamma-1}v(\xi )~\mathrm{d}\xi,
\end{align*}

\item[$\bullet$] right Riemann-Liouville fractional derivative:
\begin{align*}
{_xD_b^\gamma v(x)}=\frac{(-1)^n}{\Gamma (n-\gamma )}\frac{{{\mathrm{d} }^{n}}}
{\mathrm{d} {{x}^{n}}}\int_{x}^{b}{(\xi -x)}^{   n-\gamma-1 }v(\xi )~\mathrm{d}\xi.
\end{align*}
\end{itemize}
\end{definition}

Generally, the standard Gr\"{u}nwald-Letnikov difference formula is applied to approximate the Riemann-Liouville fractional derivative. Meerschaert and Tadjeran \cite{b10} showed that the standard Gr\"{u}nwald-Letnikov difference formula was often unstable for time dependent problems and they proposed the shifted Gr\"{u}nwald difference operators
\begin{equation*}
A^\gamma_{h,p}v(x)=\frac{1}{h^\gamma}\sum_{k=0}^{\infty}g_k^{(\gamma)}v(x-(k-p)h),\quad
B^\gamma_{h,q}v(x)=\frac{1}{h^\gamma}\sum_{k=0}^{\infty}g_k^{(\gamma)}v(x+(k-q)h),
\end{equation*}
whose accuracies are first order, i.e.,
\begin{equation*}
A^\gamma_{h,p}v(x)={_{-\infty}D_x^\gamma v(x)}+\mathcal {O}(h),\quad
B^\gamma_{h,q}v(x)={_{x}D_{+\infty}^\gamma v(x)}+\mathcal {O}(h),
\end{equation*}
where $p, q$ are integers and $g_k^{(\gamma)}=(-1)^k{\gamma \choose k}$. In fact, the coefficients $g_k^{(\gamma)}$ are the coefficients of the power series of the function $(1-z)^\gamma$,
\begin{equation*}
(1-z)^\gamma=\sum_{k=0}^{\infty}(-1)^k{\gamma \choose
k}z^k=\sum_{k=0}^{\infty}g_k^{(\gamma)}z^k,
\end{equation*}
for all $|z|\leq 1$, and they can be evaluated recursively
\begin{equation*}
g_0^{(\gamma)}=1,\quad
g_k^{(\gamma)}=(1-\frac{\gamma+1}{k})g_{k-1}^{(\gamma)},~k=1,2,\cdots
\end{equation*}

\begin{lemma}
[\cite{b13}]\label{lm1}Suppose that $0<\alpha<1$, then the coefficients
$g_k^{(\alpha)}$ satisfy
\begin{equation*}
  \begin{cases}
   ~g_0^{(\alpha)}=1,~g_1^{(\alpha)}=-\alpha<0,~g_2^{(\alpha)}=\frac{\alpha(\alpha-1)}{2}<0,\\
   ~g_1^{(\alpha)}< g_2^{(\alpha)}< g_3^{(\alpha)}<\cdots<0,\\
   ~\sum\limits_{k=0}^{\infty}g_k^{(\alpha)}=0,~\sum\limits_{k=0}^{m}g_k^{(\alpha)}>0,~m\geq1.
\end{cases}
\end{equation*}
\end{lemma}

\begin{lemma}
[\cite{b13,b16}]\label{lm2}Suppose that $1<\beta\leq2$, then the coefficients
$g_k^{(\beta)}$ satisfy
\begin{equation*}
 \begin{cases}
   ~g_0^{(\beta)}=1,~g_1^{(\beta)}=-\beta<0,~g_2^{(\beta)}=\frac{\beta(\beta-1)}{2}>0,\\
   ~1\geq g_2^{(\beta)}\geq g_3^{(\beta)}\geq\cdots\geq0,\\
   ~\sum\limits_{k=0}^{\infty}g_k^{(\beta)}=0,~\sum\limits_{k=0}^{m}g_k^{(\beta)}<0,~m\geq1.
\end{cases}
\end{equation*}
\end{lemma}
Inspired by the shifted Gr\"unwald difference operators and multi-step method, Tian et al. \cite{b16} derive the WSGD operators:
\begin{align*}
{_LD^\gamma_{h,p,q}}v(x)&=\frac{\gamma-2q}{2(p-q)}A^\gamma_{h,p}v(x)
+\frac{2p-\gamma}{2(p-q)}A^\gamma_{h,q}v(x),\\
{_RD^\gamma_{h,p,q}}v(x)&=\frac{\gamma-2q}{2(p-q)}B^\gamma_{h,p}v(x)
+\frac{2p-\gamma}{2(p-q)}B^\gamma_{h,q}v(x).
\end{align*}

\begin{lemma}
[\cite{b16}]\label{lm3} Supposing that $1<\gamma<2$, let $v(x)\in L^1(\mathbb{R})$, ${_{-\infty}D_x^\gamma v(x)}$ and
${_{x}D_{+\infty}^\gamma v(x)}$ and their Fourier transforms belong to $L^1(\mathbb{R})$, then the WSGD operators satisfy
\begin{eqnarray*}
{_LD^\gamma_{h,p,q}}v(x)={_{-\infty}D_x^\gamma v(x)}+\mathcal{O}(h^2),\quad
{_RD^\gamma_{h,p,q}}v(x)={_{x}D_{+\infty}^\gamma v(x)}+ \mathcal{O}(h^2),
\end{eqnarray*}
uniformly for $x\in \mathbb{R}$, where $p$, $q$ are integers and $p\neq q$.
\end{lemma}
Tian et al. prove Lemma \ref{lm3} under the additional conditions that $1<\gamma<2$. In fact, their proof also holds for the $0<\gamma<1$ case. The proof proceeds the same as  \cite{b16}, hence we will not repeat it here.

\begin{remark}\label{rem1}
Considering a well defined function $v(x)$ on the bounded interval $[a,b]$, if $v(a)=0$ or $v(b)=0$, the function
$v(x)$ can be zero extended for $x < a$ or $x > b$. And then the $\gamma$ order left and right Riemann-Liouville fractional derivatives of $v(x)$ at each point $x$ can be approximated by the WSGD operators with second order accuracy
\begin{align*}
{_aD_x^{\gamma}v(x)}&=\frac{\lambda_1}{h^\gamma}\sum_{k=0}^{[\frac{x-a}{h}]+p}g_k^{(\gamma)}v(x-(k-p)h)
+\frac{\lambda_2}{h^\gamma}\sum_{k=0}^{[\frac{x-a}{h}]+q}g_k^{(\gamma)}v(x-(k-q)h)+\mathcal{O}(h^2),\\
{_xD_b^{\gamma}v(x)}&=\frac{\lambda_1}{h^\gamma}\sum_{k=0}^{[\frac{b-x}{h}]+p}g_k^{(\gamma)}v(x+(k-p)h)
+\frac{\lambda_2}{h^\gamma}\sum_{k=0}^{[\frac{b-x}{h}]+q}g_k^{(\gamma)}v(x+(k-q)h)+\mathcal{O}(h^2),
\end{align*}
where $\lambda_1=\frac{\gamma-2q}{2(p-q)}$ and $\lambda_2=\frac{2p-\gamma}{2(p-q)}$.
\end{remark}
When $(p,q)=(1,0)$, $0<\alpha<1$ and $1<\beta\leq2$, the discrete approximations for the Riemann-Liouville fractional derivatives on the domain $[0,L]$ are
\begin{align} \label{eq4}
{_0D_x^{\alpha}}v(x_i)&=\frac{1}{h^\alpha}\sum_{k=0}^{i+1}w_k^{(\alpha)}v(x_{i-k+1})
+\mathcal{O}(h^2),\\\label{eq5}
{_xD_L^{\alpha}}v(x_i)&=\frac{1}{h^\alpha}\sum_{k=0}^{m-i+1}w_k^{(\alpha)}v(x_{i+k-1})
+\mathcal{O}(h^2),\\\label{eq6}
{_0D_x^{\beta}}v(x_i)&=\frac{1}{h^\beta}\sum_{k=0}^{i+1}w_k^{(\beta)}v(x_{i-k+1})
+\mathcal{O}(h^2),\\\label{eq7}
{_xD_L^{\beta}}v(x_i)&=\frac{1}{h^\beta}\sum_{k=0}^{m-i+1}w_k^{(\beta)}v(x_{i+k-1})
+\mathcal{O}(h^2),
\end{align}
where
\begin{equation}\label{eq8}
w_0^{(\alpha)}=\frac{\alpha}{2}g_0^{(\alpha)},~w_k^{(\alpha)}=\frac{\alpha}{2}g_k^{(\alpha)}
+\frac{2-\alpha}{2}g_{k-1}^{(\alpha)},~k\geq1
\end{equation}
\begin{equation}\label{eq9}
w_0^{(\beta)}=\frac{\beta}{2}g_0^{(\beta)},~w_k^{(\beta)}=\frac{\beta}{2}g_k^{(\beta)}
+\frac{2-\beta}{2}g_{k-1}^{(\beta)},~k\geq1
\end{equation}
Now, we discuss the properties of the coefficients $w_k^{(\alpha)}$ and $w_k^{(\beta)}$.
\begin{lemma}\label{lm4}
Suppose that $0<\alpha<1$, then the coefficients $w_k^{(\alpha)}$ satisfy
\begin{equation*}
  \begin{cases}
   ~w_0^{(\alpha)}=\frac{\alpha}{2}>0,~w_1^{(\alpha)}=\frac{2-\alpha-\alpha^2}{2}>0,
   ~w_2^{(\alpha)}=\frac{\alpha(\alpha^2+\alpha-4)}{4}<0,\\
   ~w_2^{(\alpha)}< w_3^{(\alpha)}< w_4^{(\alpha)}<\cdots<0,\\
   ~\sum\limits_{k=0}^{\infty}w_k^{(\alpha)}=0,~\sum\limits_{k=0}^{m}w_k^{(\alpha)}>0,~m\geq1.
\end{cases}
\end{equation*}
\begin{proof}
Combining the definition of $w_k^{(\alpha)}$ and the property of $g_k^{(\alpha)}$, it is easy to derive the value of $w_0^{(\alpha)}$, $w_1^{(\alpha)}$ and $w_2^{(\alpha)}$. When $k\geq2$, by the definition of $w_k^{(\alpha)}$
\begin{align*}
w_k^{(\alpha)}=\frac{\alpha}{2}g_k^{(\alpha)}
+\frac{2-\alpha}{2}g_{k-1}^{(\alpha)},
\end{align*}
we have $w_k^{(\alpha)}<0$ as $g_k^{(\alpha)}<0$ for $k\geq1$ and $0<\alpha<1$. Moreover,
\begin{align*}
w_{k+1}^{(\alpha)}-w_k^{(\alpha)}=\frac{\alpha}{2}(g_{k+1}^{(\alpha)}-g_k^{(\alpha)})
+\frac{2-\alpha}{2}(g_k^{(\alpha)}-g_{k-1}^{(\alpha)}).
\end{align*}
Recalling Lemma \ref{lm1}, $g_k^{(\alpha)}<g_{k+1}^{(\alpha)}$ when $k\geq1$, we obtain when $k\geq2$
$w_{k+1}^{(\alpha)}-w_k^{(\alpha)}>0$, i.e.,
$w_k^{(\alpha)}<w_{k+1}^{(\alpha)}$.
For the sum $\sum\limits_{k=0}^{\infty}w_k^{(\alpha)}$, we have
\begin{align*}
\sum\limits_{k=0}^{\infty}w_k^{(\alpha)}&=w_0^{(\alpha)}+\sum\limits_{k=1}^{\infty}w_k^{(\alpha)}=\frac{\alpha}{2}g_0^{(\alpha)}+\sum\limits_{k=1}^{\infty}(\frac{\alpha}{2}g_k^{(\alpha)}
+\frac{2-\alpha}{2}g_{k-1}^{(\alpha)})\\
&=\frac{\alpha}{2}\sum\limits_{k=0}^{\infty}g_k^{(\alpha)}+\frac{2-\alpha}{2}\sum\limits_{k=0}^{\infty}g_k^{(\alpha)}=\sum\limits_{k=0}^{\infty}g_k^{(\alpha)}=0.
\end{align*}
Since $w_k^{(\alpha)}<0$ when $k\geq2$,
$\sum\limits_{k=0}^{m}w_k^{(\alpha)}>0$ for $m\geq1$.
\end{proof}
\end{lemma}
\begin{lemma}
[\cite{b16}]\label{lm5}Suppose that $1<\beta\leq2$, then the coefficients $w_k^{(\beta)}$ satisfy
\begin{equation*}
 \begin{cases}
   ~w_0^{(\beta)}=\frac{\beta}{2}>0,~w_1^{(\beta)}=\frac{2-\beta-\beta^2}{2}<0,~w_2^{(\beta)}
   =\frac{\beta(\beta^2+\beta-4)}{4},\\
   ~1\geq w_0^{(\beta)}\geq w_3^{(\beta)}\geq w_4^{(\beta)}\geq\cdots\geq0,\\
   ~\sum\limits_{k=0}^{\infty}w_k^{(\beta)}=0,~\sum\limits_{k=0}^{m}w_k^{(\beta)}<0,~m\geq2.
\end{cases}
\end{equation*}
\end{lemma}
\section{The finite difference method for the RFADE}\label{sec3}
In this section, we utilize the Eqs.(\ref{eq4})-(\ref{eq7}) to approximate the Riesz space fractional derivative and derive the Crank-Nicolson scheme of the equation. We define $t_n=n\tau$, $n=0, 1,\cdots, N$, let $\Omega=[0,L]$ be a {f}inite domain, setting $S_h$ be a uniform partition of $\Omega$, which is given by $x_i=ih$ for $i=0, 1,\cdots,m$, where $\tau=T/N$ and $h=L/m$ are the time and space steps, respectively. At point $(x_i,t_{n-\frac{1}{2}})$, it is easy to conclude that,
\begin{eqnarray*}
\frac{u(x_i,t_n)-u(x_i,t_{n-1})}{\tau}=\bigg(\frac{\partial
u(x,t)}{\partial t}\bigg)_i^{n-1/2}+\mathcal {O}(\tau^2).
\end{eqnarray*}
\begin{align*}
&\bigg(K_{\alpha}\frac{\partial^\alpha u(x,t)}{\partial |x|^\alpha}+K_{\beta}\frac{\partial^\beta u(x,t)}{\partial|x |^\beta} \bigg)_i^{n-1/2}\\
=&\frac{1}{2}\bigg(K_{\alpha}\frac{\partial^\alpha u(x_i,t_n)}{\partial| x |^\alpha}+K_{\beta}\frac{\partial^\beta u(x_i,t_n)}{\partial|x|^\beta}\bigg)+\frac{1}{2}\bigg(K_{\alpha}\frac{\partial^\alpha u(x_i,t_{n-1})}{\partial|x|^\alpha}+K_{\beta}\frac{\partial^\beta u(x_i,t_{n-1})}{\partial|x|^\beta} \bigg)+\mathcal
{O}(\tau^2).
\end{align*}
We present the semi-discrete form of Eq.(\ref{eq1}),
\begin{align}
&\frac{u(x_i,t_n)-u(x_i,t_{n-1})}{\tau}\nonumber\\\label{eq10}
=&\frac{1}{2}\bigg\{K_{\alpha}\frac{\partial^\alpha u(x_i,t_n)}{\partial | x|^\alpha}
+K_{\beta}\frac{\partial^\beta u(x_i,t_n)}{\partial |x|^\beta}\bigg\}
+\frac{1}{2}\bigg\{K_{\alpha}\frac{\partial^\alpha u(x_i,t_{n-1})}{\partial |x|^\alpha}
+K_{\beta}\frac{\partial^\beta u(x_i,t_{n-1})}{\partial |x|^\beta}\bigg\}+O(\tau^2).
\end{align}
Substituting (\ref{eq4})-(\ref{eq7}) into (\ref{eq10}), we obtain
\begin{align}
\frac{u(x_i,t_n)-u(x_i,t_{n-1})}{\tau}
=&-\frac{K_\alpha c_\alpha}{2h^\alpha}\bigg[\sum_{k=0}^{i+1}w_k^{(\alpha)}u(x_{i-k+1},t_{n})
+\sum_{k=0}^{m-i+1}w_k^{(\alpha)}u(x_{i+k-1},t_{n})\bigg]\nonumber\\
-&\frac{ K_\beta c_\beta}{2h^\beta}\bigg[\sum_{k=0}^{i+1}w_k^{(\beta)}u(x_{i-k+1},t_{n})
+\sum_{k=0}^{m-i+1}w_k^{(\beta)}u(x_{i+k-1},t_{n})\bigg]\nonumber\\ 
-&\frac{K_\alpha c_\alpha}{2h^\alpha}\bigg[\sum_{k=0}^{i+1}w_k^{(\alpha)}u(x_{i-k+1},t_{n-1})
+\sum_{k=0}^{m-i+1}w_k^{(\alpha)}u(x_{i+k-1},t_{n-1})\bigg]\nonumber\\ \label{eqN1}
-&\frac{ K_\beta c_\beta}{2h^\beta}\bigg[\sum_{k=0}^{i+1}w_k^{(\beta)}u(x_{i-k+1},t_{n-1})
+\sum_{k=0}^{m-i+1}w_k^{(\beta)}u(x_{i+k-1},t_{n-1})\bigg]+O(\tau^2+h^2).
\end{align}
Let $u_i^n$ be the approximation solution of $u(x_i,t_n)$, then we can obtain the numerical scheme
\begin{align}\label{eq11}
&u_i^n+\mu_\alpha\bigg[\sum_{k=0}^{i+1}w_k^{(\alpha)}u^n_{i-k+1}+\sum_{k=0}^{m-i+1}w_k^{(\alpha)}u^n_{i+k-1}\bigg]
+\mu_\beta\bigg[\sum_{k=0}^{i+1}w_k^{(\beta)}u^n_{i-k+1}+\sum_{k=0}^{m-i+1}w_k^{(\beta)}u^n_{i+k-1}\bigg]\nonumber\\
=&u_i^{n-1}-\mu_\alpha\bigg[\sum_{k=0}^{i+1}w_k^{(\alpha)}u^{n-1}_{i-k+1}+\sum_{k=0}^{m-i+1}w_k^{(\alpha)}u^{n-1}_{i+k-1}\bigg]
-\mu_\beta\bigg[\sum_{k=0}^{i+1}w_k^{(\beta)}u^{n-1}_{i-k+1}+\sum_{k=0}^{m-i+1}w_k^{(\beta)}u^{n-1}_{i+k-1}\bigg],
\end{align}
where $\mu_\alpha=\frac{\tau K_\alpha c_\alpha}{2h^\alpha}$ and $\mu_\beta=\frac{\tau K_\beta c_\beta}{2h^\beta}$. Denote
\begin{equation}\label{eq12}
A=\left(\begin{array}{cccccc}
w_1^{(\alpha)} & w_0^{(\alpha)} & 0 & \cdots & 0 & 0\\
w_2^{(\alpha)} & w_1^{(\alpha)} & w_0^{(\alpha)} & \cdots & 0 & 0\\
w_3^{(\alpha)} & w_2^{(\alpha)} & w_1^{(\alpha)} & \cdots & 0 & 0\\
\vdots & \vdots & \vdots & \ddots & \vdots & \vdots\\
w_{m-2}^{(\alpha)} & w_{m-3}^{(\alpha)} & w_{m-4}^{(\alpha)} & \cdots & w_1^{(\alpha)} & w_0^{(\alpha)}\\
w_{m-1}^{(\alpha)} & w_{m-2}^{(\alpha)} & w_{m-3}^{(\alpha)} &
\cdots & w_2^{(\alpha)} & w_1^{(\alpha)}
\end{array}\right),
\end{equation}
\begin{equation}\label{eq13}
B=\left(\begin{array}{cccccc}
w_1^{(\beta)} & w_0^{(\beta)} & 0 & \cdots & 0 & 0\\
w_2^{(\beta)} & w_1^{(\beta)} & w_0^{(\beta)} & \cdots & 0 & 0\\
w_3^{(\beta)} & w_2^{(\beta)} & w_1^{(\beta)} & \cdots & 0 & 0\\
\vdots & \vdots & \vdots & \ddots & \vdots & \vdots\\
w_{m-2}^{(\beta)} & w_{m-3}^{(\beta)} & w_{m-4}^{(\beta)} & \cdots & w_1^{(\beta)} & w_0^{(\beta)}\\
w_{m-1}^{(\beta)} & w_{m-2}^{(\beta)} & w_{m-3}^{(\beta)} & \cdots &
w_2^{(\beta)} & w_1^{(\beta)}
\end{array}\right),
\end{equation}
and
\begin{equation}\label{eq14}
D=\mu_\alpha(A+A^T)+\mu_\beta(B+B^T),
\end{equation}
\begin{eqnarray*}
U^{n}=[u_1^n,u_2^n,\cdots,u_{m-1}^n]^{T}.
\end{eqnarray*}
Thus, Eq. (\ref{eq11}) can be simplified as
\begin{eqnarray}\label{eq15}
(I+D) U^n=(I-D)U^{n-1}.
\end{eqnarray}
The boundary and initial conditions are discretized as
\begin{equation*}
u_i^0 = \psi (ih),\quad U^0=[u^0_1,u^0_2,\cdots,u^0_{m-1}]^T,
\end{equation*}
where $i=1, 2, \cdots,m-1$.

\section{Theoretical analysis of the {f}inite difference method}\label{sec4}

\subsection{Stability}\label{sec4.1}
Before giving the proof, we start with some useful lemmas.
\begin{lemma}[\cite{b17}]\label{lm6} Let $\mathscr{A}$ be an $m-1$ order positive define
matrix, then for any parameter $\theta\geq0$, the following two
inequalities
\begin{eqnarray*}
||(I+\theta \mathscr{A})^{-1}||\leq1,\quad
||(I+\theta \mathscr{A})^{-1}(I-\theta \mathscr{A})||\leq1
\end{eqnarray*}
hold.
\end{lemma}
Now, we discuss the property of matrix $D$.
\begin{theorem}\label{thm1}   
Suppose that $0<\alpha<1$ and $1<\beta\leq2$, $A$, $B$ and $D$ are defined as (\ref{eq12})-(\ref{eq14}), then the coefficients $D_{ij}$ satisfy
\begin{eqnarray*}
|D_{ii}|>\sum_{j=1,j\neq i}^{m-1}|D_{ij}|,\quad i=1,2,\cdots,m-1.
\end{eqnarray*}
i.e., $D$ is strictly diagonally dominant.
\end{theorem}
\begin{proof}
It is easy to obtain
\begin{eqnarray*}
 D_{ij}=
  \begin{cases}
   \mu_\alpha w_{j-i+1}^{(\alpha)}+\mu_\beta w_{j-i+1}^{(\beta)} ,&\mbox{$j> i+1$,}\\
   \mu_\alpha( w_{0}^{(\alpha)}+ w_{2}^{(\alpha)})+\mu_\beta( w_{0}^{(\beta)}+ w_{2}^{(\beta)}) ,&\mbox{$j=i+1$,}\\
    2\mu_\alpha w_{1}^{(\alpha)}+2\mu_\beta w_{1}^{(\beta)},&\mbox{$j=i$,}\\
   \mu_\alpha( w_{0}^{(\alpha)}+ w_{2}^{(\alpha)})+\mu_\beta( w_{0}^{(\beta)}+ w_{2}^{(\beta)}),&\mbox{$j= i-1$,}\\
  \mu_\alpha w_{i-j+1}^{(\alpha)}+\mu_\beta w_{i-j+1}^{(\beta)}.&\mbox{$j<i-1$,}
\end{cases}
\end{eqnarray*}
where $\mu_\alpha=\frac{\tau K_\alpha c_\alpha}{2h^\alpha}>0$ and $\mu_\beta=\frac{\tau K_\beta c_\beta}{2h^\beta}<0$. First, we consider the signs of $D_{ij}$. According to Lemma \ref{lm4}, when $k\geq 3$, $w_k^{(\alpha)}<0$, thus $\mu_\alpha w_k^{(\alpha)}<0$. According to Lemma \ref{lm5}, when $k\geq 3$, $w_k^{(\beta)}>0$, thus $\mu_\beta w_k^{(\beta)}<0$. Therefore, $D_{ij}<0$ when $j>i+1$ or $j<i-1$. For the items $D_{i,i+1}$ and $D_{i,i-1}$, we have
\begin{align*}
w_{0}^{(\alpha)}+w_{2}^{(\alpha)}&=\frac{\alpha}{2}+\frac{\alpha(\alpha^2+\alpha-4)}{4}=\frac{\alpha(\alpha+2)(\alpha-1)}{4}<0,\\
w_{0}^{(\beta)}+w_{2}^{(\beta)}&=\frac{\beta}{2}+\frac{\beta(\beta^2+\beta-4)}{4}=\frac{\beta(\beta+2)(\beta-1)}{4}>0.
\end{align*}
Since $\mu_\alpha>0$ and $\mu_\beta<0$, then
\begin{eqnarray*}
D_{i,i+1}=D_{i,i-1}=\mu_\alpha( w_{0}^{(\alpha)}+
w_{2}^{(\alpha)})+\mu_\beta( w_{0}^{(\beta)}+ w_{2}^{(\beta)})<0.
\end{eqnarray*}
According to Lemmas \ref{lm4} and Lemma \ref{lm5}, we have $w_{1}^{(\alpha)}>0$ and $w_{1}^{(\beta)}<0$, thus
\begin{eqnarray*}
D_{i,i}=2\mu_\alpha w_{1}^{(\alpha)}+2\mu_\beta w_{1}^{(\beta)}>0,
\end{eqnarray*}
as $\mu_\alpha>0$ and $\mu_\beta<0$. Now, for a given $i$, we consider the sum
\begin{align*}
&\sum_{j=1,j\neq
i}^{m-1}|D_{ij}|=\sum_{j=1}^{i-2}|D_{ij}|+\sum_{j=i+2}^{m-1}|D_{ij}|+|D_{i,i-1}|+|D_{i,i+1}|\\
=&-\sum_{j=1}^{i-2}(\mu_\alpha w_{i-j+1}^{(\alpha)}+\mu_\beta
w_{i-j+1}^{(\beta)})-\sum_{j=i+2}^{m-1}( \mu_\alpha
w_{j-i+1}^{(\alpha)}+\mu_\beta w_{j-i+1}^{(\beta)})- 2\mu_\alpha( w_{0}^{(\alpha)}+ w_{2}^{(\alpha)})-2\mu_\beta( w_{0}^{(\beta)}+ w_{2}^{(\beta)})\\
<&\sum_{j=-\infty}^{i-2}(\mu_\alpha w_{i-j+1}^{(\alpha)}+\mu_\beta
w_{i-j+1}^{(\beta)})-\sum_{j=i+2}^{+\infty}(\mu_\alpha
w_{j-i+1}^{(\alpha)}+\mu_\beta w_{j-i+1}^{(\beta)})- 2\mu_\alpha( w_{0}^{(\alpha)}+ w_{2}^{(\alpha)})-2\mu_\beta(
w_{0}^{(\beta)}+ w_{2}^{(\beta)})\\
=&-2\mu_\alpha\sum_{k=3}^{+\infty}w_{k}^{(\alpha)}-2\mu_\beta\sum_{k=3}^{+\infty}w_{k}^{(\beta)}
- 2\mu_\alpha( w_{0}^{(\alpha)}+ w_{2}^{(\alpha)})-2\mu_\beta( w_{0}^{(\beta)}+ w_{2}^{(\beta)})\\
=&2\mu_\alpha w_{1}^{(\alpha)}+2\mu_\beta
w_{1}^{(\beta)}-2\mu_\alpha\sum_{k=0}^{+\infty}w_{k}^{(\alpha)}-2\mu_\beta\sum_{k=0}^{+\infty}w_{k}^{(\beta)}\\
=&2\mu_\alpha w_{1}^{(\alpha)}+2\mu_\beta w_{1}^{(\beta)}=|D_{i,i}|
\end{align*}
i.e., $$\sum_{j=1,j\neq i}^{m-1}|D_{ij}|<|D_{ii}|.$$
Thus, the proof is completed.
\end{proof}
According to the theorem, it is easy to conclude the following corollaries.
\begin{corollary}\label{col1}
The matrix $I+D$ is strictly diagonally dominant as well. Therefore, $I+D$ is invertible and Eq. (\ref{eq15}) is solvable.
\end{corollary}
\begin{corollary} \label{col2} 
The matrix $D$ is symmetric positive definite.
\end{corollary}
\begin{proof}
In view of (\ref{eq14}), the symmetry of $D$ is evident. Let $\lambda_0$ be one eigenvalue of the matrix $D$. Then by the Gerschgorin's circle theorem \cite{b30}, we have
\begin{eqnarray*}
|\lambda_0-D_{ii}|\leq r_i=\sum_{j=1,j\neq i}^{m-1}|D_{ij}|,
\end{eqnarray*}
i.e.,
\begin{eqnarray*}
D_{ii}-\sum_{j=1,j\neq i}^{m-1}|D_{ij}|\leq\lambda_0\leq D_{ii}+\sum_{j=1,j\neq i}^{m-1}|D_{ij}|.
\end{eqnarray*}
In view of Theorem \ref{thm1}, we have $\lambda_0>0$, thus $D$ is positive definite, which completes the proof.
\end{proof}
Since $D$ is symmetric positive definite, by Lemma \ref{lm6}, the following two inequalities hold.
\begin{corollary}\label{col3}  
\begin{eqnarray*}
||(I+D)^{-1}||\leq1,\quad ||(I+D)^{-1}(I-D)||\leq1.
\end{eqnarray*}
\end{corollary}

\begin{theorem}\label{thm2}
The difference scheme (\ref{eq15}) is unconditionally stable.
\end{theorem}
\begin{proof}
Let $U^n$ and $u^n$ be the numerical and exact solution vectors respectively and
$u^n=[\,u(x_1,t_n)$, $u(x_2,t_n),\cdots$, $u(x_{m-1},t_n)\,]$. Since the
matrix $(I + D)$ is invertible, then we can obtain the following
error equation
\begin{equation}\label{eq16}
\mathscr{E}\,^n=M\mathscr{E}\,^{n-1},
\end{equation}
where $\mathscr{E}\,^n=U^n-u^n$ and $M=(I+D)^{-1}(I-D)$. By (\ref{eq16}) we have
\begin{equation*}
\mathscr{E}\,^n=M^{n-1}\mathscr{E}\,^{0},
\end{equation*}
Applying Corollary \ref{col3}, we obtain
\begin{equation*}
||\mathscr{E}\,^n||\leq
||M^{n-1}|| ||\mathscr{E}\,^{0}||\leq
||M||^{n-1}||\mathscr{E}\,^{0}||\leq||\mathscr{E}\,^{0}||,
\end{equation*}
which means difference scheme (\ref{eq15}) is unconditionally stable.
\end{proof}
\subsection{Convergence} \label{sec4.2}
In the following, we suppose the symbol $C$ is a generic positive constant, which may take different values at different places. According to (\ref{eqN1}), the local truncation error of the Crank-Nicolson scheme (\ref{eq11}) is $\mathcal {R}_i^n=\mathcal {O}(\tau^3+\tau h^2)$.


\begin{theorem}\label{thm3}
The numerical solution $U^n$ unconditionally converges to the exact solution $u^n$ as $h$ and $\tau $ tend to zero, and
\begin{eqnarray*}
||U^n-u^n||\leq C(\tau^2 +h^2).
\end{eqnarray*}
\end{theorem}
\begin{proof}
Let $e^n_i$ denote the error at grid points $(x_i,t_n)$ and $e^n_i=u(x_i,t_n)-u^n_i$. Combining Eqs.(\ref{eqN1})-(\ref{eq11}), yields
\begin{align*}
&e_{i}^{n}+\mu_\alpha\bigg[\sum_{k=0}^{i+1}w_k^{(\alpha)}e^n_{i-k+1}+\sum_{k=0}^{m-i+1}w_k^{(\alpha)}e^n_{i+k-1}\bigg]
+\mu_\beta\bigg[\sum_{k=0}^{i+1}w_k^{(\beta)}e^n_{i-k+1}+\sum_{k=0}^{m-i+1}w_k^{(\beta)}e^n_{i+k-1}\bigg]\nonumber\\
=&e_i^{n-1}-\mu_\alpha\bigg[\sum_{k=0}^{i+1}w_k^{(\alpha)}e^{n-1}_{i-k+1}+\sum_{k=0}^{m-i+1}w_k^{(\alpha)}e^{n-1}_{i+k-1}\bigg]
-\mu_\beta\bigg[\sum_{k=0}^{i+1}w_k^{(\beta)}e^{n-1}_{i-k+1}+\sum_{k=0}^{m-i+1}w_k^{(\beta)}e^{n-1}_{i+k-1}\bigg]+\mathcal
{O}(\tau^3+\tau h^2).
\end{align*}
Using the conditions (\ref{eq2}) and (\ref{eq3}), we obtain the errors $e^0_i=0$ and $e^n_0=e^n_m=0$ for
$i=1,2,\cdots,m-1$ and $j=0,1,\cdots,N$. We can write the system in matrix-vector form
\begin{eqnarray*}
(I+D) E^n=(I-D)E^{n-1}+\mathcal {O}(\tau^3+\tau h^2)\chi
\end{eqnarray*}
or
\begin{eqnarray*}
E^n=ME^{n-1}+\bm{b},
\end{eqnarray*}
where $\chi=[1,1,\cdots,1]^T$, $E^n=[e^n_1,e^n_2,\cdots,e^n_{m-1}]^T$, $D=\mu_\alpha(A+A^T)+\mu_\beta(B+B^T)$, $M=(I+D)^{-1}(I-D)$ and $\bm{b}=\mathcal {O}(\tau^3+\tau h^2)(I+D)^{-1}$. By iterating and noting that $E^0=\bm{0}$, we obtain
\begin{eqnarray*}
E^n=(M^{n-1}+M^{n-2}+\cdots+I)\bm{b}.
\end{eqnarray*}
Now, from Corollary \ref{col3}, we have $||M||<1$ and $||(I+D)^{-1}||<1$. Then upon taking norms,
\begin{align*}
||E^n||&\leq(||M^{n-1}||+||M^{n-2}||+\cdots+1)||\bm{b}||\\
&\leq(1+1+\cdots+1)||\bm{b}||\\
&\leq n \mathcal {O}(\tau^3+\tau h^2)=T\mathcal {O}(\tau^2 +h^{2}).
\end{align*}
Thus,
\begin{align*}
||E^n||_{\infty}\leq C(\tau^2+h^{2}),
\end{align*}
which completes the proof.
\end{proof}
\subsection{Improving the convergence order}\label{sec4.3}
Here we use the Richardson extrapolation method (REM) to improve the convergence order. Suppose that $\mathcal {I}_h f$ is the approximation solution of the function of $f(x)$ with an asymptotic expansion
\begin{eqnarray*}
f=\mathcal {I}_h f+C_1h^2+\mathcal {O}(h^3),\quad h\to 0,~C_1\neq0.
\end{eqnarray*}
Then we have
\begin{eqnarray*}
f=\mathcal {I}_{h/2} f+C_1(h/2)^2+\mathcal {O}(h^3).
\end{eqnarray*}
Eliminating the middle terms $C_1h^2$ on the right, we find
\begin{eqnarray}\label{eq17}
f=\frac{4\mathcal {I}_{h/2} f-\mathcal {I}_h f}{3}+\mathcal
{O}(h^3),
\end{eqnarray}
which means the approximation order of $f(x)$ has been improved from $\mathcal {O}(h^2)$ to $\mathcal {O}(h^3)$. Repeatedly, we can improve the approximation order of $f(x)$ from $\mathcal {O}(h^3)$ to $\mathcal {O}(h^4)$. Suppose that $\mathcal {G}_h f$ is the approximation solution of the function of $f(x)$ with an asymptotic expansion
\begin{eqnarray*}
f=\mathcal {G}_h f+C_2h^3+\mathcal {O}(h^4),\quad h\to 0,~C_2\neq0.
\end{eqnarray*}
Then we have
\begin{eqnarray*}
f=\mathcal {G}_{h/2} f+C_2(h/2)^3+\mathcal {O}(h^4).
\end{eqnarray*}
Eliminating the middle terms $C_2h^3$ on the right, we find
\begin{eqnarray}\label{eq18}
f=\frac{8\mathcal {G}_{h/2} f-\mathcal {G}_h f}{7}+\mathcal
{O}(h^4).
\end{eqnarray}
According to Theorem \ref{thm3}, we know that the numerical method converges at the rate of $\mathcal{O}(\tau^2+h^2)$. In order to improve the convergence order, we apply the REM on a coarse grid $\tau=h$ and then on a finer grid of size $\tau/2=h/2$ and $\tau/4=h/4$. By applying the Richardson extrapolation formulae (\ref{eq17}) and (\ref{eq18}) consecutively, the convergence order can be improved from $\mathcal {O}(\tau^2+h^2)$ to $\mathcal {O}(\tau^4+h^4)$.

\section{Numerical examples}\label{sec5}
In order to demonstrate the effectiveness of numerical methods, some examples are presented.
\begin{example} First, we consider the following RFDE ($K_\alpha=0$) with a source term
\begin{eqnarray*}
   \begin{cases}
   ~\frac{\partial u(x,t)}{\partial {t}}=\frac{\partial^\beta u(x,t)}{\partial\left|x\right|^\beta}+f(x,t),\quad 0<x<1,\quad 0<t\leq T, \\
   ~u(x,0)=x^2(1-x)^2,\quad\quad 0\le x\le 1,\\
   ~ u(0,t)=u(1,t)=0,\quad\quad 0\le t\le T,
\end{cases}
\end{eqnarray*}
where $1<\beta\leq 2$,
\begin{align*}
f(x,t)&=-x^2(1-x)^2e^{-t}+\frac{e^{-t}}{2\cos\frac{\beta\pi}{2}}
\Big\{\frac{24}{\Gamma(5-\beta)}[x^{4-\beta}+(1-x)^{4-\beta}]\\
&-\frac{12}{\Gamma(4-\beta)}[x^{3-\beta}+(1-x)^{3-\beta}]
+\frac{2}{\Gamma(3-\beta)}[x^{2-\beta}+(1-x)^{2-\beta}]\Big\},
\end{align*}
and the exact solution is $u(x,t)=x^2(1-x)^2e^{-t}$.
\end{example}
The related numerical results are given in Table \ref{tab1}. It shows the error and convergence order of the Crank-Nicolson scheme at $t=1$ with $\tau=h$, where $\beta$ is corresponding to three distinct values: $\beta=1.2$, $\beta=1.5$, $\beta=1.8$. It can be seen that the numerical results are in excellent agreement with the exact solution.
\begin{table}[H]
\caption{The error and convergence order of the CN scheme of the RFDE.}
\label{tab1}
\centering
\begin{tabular}{c  c c c ccc}
 \toprule
 \multirow{2}{*}{\centering \quad$\tau=h$\quad}
 &\multicolumn{2}{c}{$ \beta=1.2$} &\multicolumn{2}{c}{$ \beta=1.5$} &\multicolumn{2}{c}{$ \beta=1.8$}\\
\cmidrule{2-7}
        & $ ||E(h,\tau)||$ &  Order  &  $ ||E(h,\tau)||$  &  Order  &  $ ||E(h,\tau)||$  &  Order  \\
\midrule
$ 1/8$  & 1.6781E-03     &        &  1.8350E-03      &        &  1.7827E-03      &       \\
$1/16$  & 3.9728E-04     &  2.08  &  4.3608E-04      &  2.07  &  4.3250E-04      & 2.04  \\
$1/32$  & 9.4291E-05     &  2.07  &  1.0349E-04      &  2.08  &  1.0482E-04      & 2.04  \\
$1/64$  & 2.2442E-05     &  2.07  &  2.4538E-05      &  2.08  &  2.5351E-05      & 2.05  \\
$1/128$ & 5.3679E-06     &  2.06  &  5.8261E-06      &  2.07  &  6.1256E-06      & 2.05  \\
\bottomrule
\end{tabular}
\end{table}
\begin{table}[H] 
\caption{The error and convergence order of the CN scheme of the RFADE.}
\label{tab2}
\centering
\begin{tabular}{c  c c c ccc}
\toprule
\multirow{2}{*}{\centering $\alpha=0.1,\tau=h$}
 &\multicolumn{2}{c}{$\beta=1.2$} &\multicolumn{2}{c}{$ \beta=1.5$} &\multicolumn{2}{c}{$ \beta=1.8$}\\
\cmidrule{2-7}
&$ ||E(h,\tau)||$&Order&  $ ||E(h,\tau)||$  &Order&  $ ||E(h,\tau)||$  &Order  \\
\midrule
$ 1/8$ & 2.8322E-05      &        &  2.8689E-05      &        & 2.4212E-05      &       \\
$1/16$  & 7.1828E-06     &  1.98  &  7.0817E-06      &  2.02  & 5.7905E-06      & 2.06  \\
$1/32$  & 1.8334E-06     &  1.97  &  1.7860E-06      &  1.99  & 1.4386E-06      & 2.01  \\
$1/64$ & 4.6575E-07      &  1.98  &  4.5086E-07      &  1.99  & 3.5998E-07      & 2.00  \\
$1/128$ & 1.1756E-07     &  1.99  &  1.1343E-07      &  1.99  & 9.0132E-08      & 2.00  \\
\midrule
 \multirow{2}{*}{\centering $\alpha=0.5,\tau=h$}
 &\multicolumn{2}{c}{$ \beta=1.2$} &\multicolumn{2}{c}{$ \beta=1.5$} &\multicolumn{2}{c}{$ \beta=1.8$}\\
\cmidrule{2-7}
& $ ||E(h,\tau)||$ &  Order  &  $ ||E(h,\tau)||$  &  Order  &  $ ||E(h,\tau)||$  &  Order  \\
\midrule
$ 1/8$  & 4.2703E-05     &        &  4.3144E-05      &        & 3.6635E-05      &       \\
$1/16$  & 1.0824E-05     &  1.98  &  1.0665E-05      &  2.02  & 8.7850E-06      & 2.06  \\
$1/32$  & 2.7622E-06     &  1.97  &  2.6913E-06      &  1.99  & 2.1851E-06      & 2.01  \\
$1/64$  & 7.0166E-07     &  1.98  &  6.7960E-07      &  1.99  & 5.4714E-07      & 2.00  \\
$1/128$ & 1.7712E-07     &  1.99  &  1.7101E-07      &  1.99  & 1.3704E-07      & 2.00  \\
\midrule
\multirow{2}{*}{\centering $\alpha=0.9,\tau=h$}
 &\multicolumn{2}{c}{$ \beta=1.2$} &\multicolumn{2}{c}{$ \beta=1.5$} &\multicolumn{2}{c}{$ \beta=1.8$}\\
\cmidrule{2-7}
& $ ||E(h,\tau)||$ &  Order  &  $ ||E(h,\tau)||$  &  Order  &  $ ||E(h,\tau)||$  &  Order  \\
\midrule
$ 1/8$  & 6.6633E-05     &        &  6.6253E-05      &        & 5.6545E-05      &       \\
$1/16$  & 1.6852E-05     &  1.98  &  1.6414E-05      &  2.01  & 1.3633E-05      & 2.05  \\
$1/32$  & 4.3006E-06     &  1.97  &  4.1491E-06      &  1.98  & 3.4009E-06      & 2.00  \\
$1/64$  & 1.0929E-06     &  1.98  &  1.0488E-06      &  1.98  & 8.5302E-07      & 2.00 \\
$1/128$ & 2.7596E-07     &  1.99  &  2.6409E-07      &  1.99  & 2.1385E-07      & 2.00  \\
\bottomrule
\end{tabular}
\end{table}

\begin{example} Then, we consider the following RFADE with a source term:
\begin{eqnarray*}
   \begin{cases}
   ~\frac{\partial u(x,t)}{\partial {t}}=K_{\alpha}\frac{\partial^\alpha}{\partial \left|x\right|^\alpha}u(x,t)
   +K_{\beta}\frac{\partial^\beta}{\partial \left|x\right|^\beta}u(x,t)+f(x,t),\quad 0<x<1,\quad 0<t\leq T, \\
   ~u(x,0)=0,\quad\quad 0\le x\le 1,\\
   ~ u(0,t)=u(1,t)=0,\quad 0\le t\le T,
\end{cases}
\end{eqnarray*}
where $0<\alpha<1$, $1<\beta\leq 2$, $l(x,p)=x^p+(1-x)^p$ and
\begin{align*}
f(x,t)&=\frac{K_\alpha t^\beta e^{\alpha t}}
{2\cos(\alpha\pi/2)}\Big\{\frac{\Gamma(7)l(x,6-\alpha)}{\Gamma(7-\alpha)}
-\frac{6\Gamma(8)l(x,7-\alpha)}{\Gamma(8-\alpha)}
+\frac{15\Gamma(9)l(x,8-\alpha)}{\Gamma(9-\alpha)}
\\&-\frac{20\Gamma(10)l(x,9-\alpha)}{\Gamma(10-\alpha)}
+\frac{15\Gamma(11)l(x,10-\alpha)}{\Gamma(11-\alpha)}
-\frac{6\Gamma(12)l(x,11-\alpha)}{\Gamma(12-\alpha)}\\
&+\frac{\Gamma(13)l(x,12-\alpha)}{\Gamma(13-\alpha)}\Big\}
+\frac{K_\beta t^\beta e^{\alpha t}}
{2\cos(\beta\pi/2)}\Big\{\frac{\Gamma(7)l(x,6-\beta)}{\Gamma(7-\beta)}
-\frac{6\Gamma(8)l(x,7-\beta)}{\Gamma(8-\beta)}
\\&+\frac{15\Gamma(9)l(x,8-\beta)}{\Gamma(9-\beta)}
-\frac{20\Gamma(10)l(x,9-\beta)}{\Gamma(10-\beta)}
+\frac{15\Gamma(11)l(x,10-\beta)}{\Gamma(11-\beta)}
\\&-\frac{6\Gamma(12)l(x,11-\beta)}{\Gamma(12-\beta)}
+\frac{\Gamma(13)l(x,12-\beta)}{\Gamma(13-\beta)}\Big\}
+t^{\beta-1}e^{\alpha t}(\beta+\alpha t)x^6(1-x)^6
\end{align*}
and the exact solution is $u(x,t)=t^{\beta}e^{\alpha t}x^6(1-x)^6$.
\end{example}
Here, we take $K_\alpha=K_\beta=2$, $\alpha=0.1, ~0.5,~0.9$ and $\beta=1.2,~1.5,~1.8$. The error and convergence order of the Crank-Nicolson scheme of the RFADE at $t=1$ with $\tau=h$ is shown in Table \ref{tab2}. As expected, the convergence order can be reached second order in both time and space direction. By applying the REM (\ref{eq17}) and (\ref{eq18}) consecutively, we get the improved error and convergence order which is shown in Table \ref{tab3}. It can be seen that the convergence order has been improved from second order to fourth order.

\begin{table}[H] 
\caption{The error and convergence order of REM.}
\label{tab3}
\centering
\begin{tabular}{c  c c c ccc}
 \toprule
 \multirow{2}{*}{\centering $\beta=1.8,\tau=h$}
 &\multicolumn{2}{c}{$ \alpha=0.1$} &\multicolumn{2}{c}{$ \alpha=0.5$} &\multicolumn{2}{c}{$ \alpha=0.9$}\\
\cmidrule{2-7}
& $ ||E(h,\tau)||$ &  Order  &  $ ||E(h,\tau)||$  &  Order  &  $ ||E(h,\tau)||$  &  Order  \\
\midrule
$ 1/8$  & 2.5866E-08     &        &  3.9130E-08      &        & 6.2601E-08      &       \\
$1/16$  & 1.7240E-09     &  3.91  &  2.6351E-09      &  3.89  & 4.3563E-09      & 3.85  \\
$1/32$  & 1.1610E-10     &  3.89  &  1.7772E-10      &  3.89  & 2.9783E-10      & 3.87  \\
$1/64$  & 7.6467E-12     &  3.92  &  1.1768E-11      &  3.92  & 1.9906E-11      & 3.90  \\
\bottomrule
\end{tabular}
\end{table}

\begin{example} Finally, we consider the following RFADE:
\begin{eqnarray*}
\begin{cases}
~\frac{\partial u(x,t)}{\partial {t}}=K_{\alpha}\frac{\partial^\alpha u(x,t)}{\partial\left|x\right|^\alpha}+K_{\beta}\frac{\partial^\beta u(x,t)}{\partial\left|x\right|^\beta},\quad 0<x<\pi,\quad 0<t\leq T, \\
~u(x,0)=x^2(\pi-x),\quad\quad 0\le x\le \pi,\\
~ u(0,t)=u(\pi,t)=0,\quad\quad0\le t\le T,
\end{cases}
\end{eqnarray*}
where $0<\alpha<1$, $1<\beta\leq 2$. According to \cite{b28}, the analytical solution is given by
\begin{equation*}
u(x,t)=\sum_{n=1}^{\infty}\Big[\frac{8}{n^3}(-1)^{n+1}-\frac{4}{n^3}\Big]\sin(nx)
\exp\Big(-[K_{\alpha}(n^2)^{\alpha/2}+K_\beta(n^2)^{\beta/2}]t\Big).
\end{equation*}
\end{example}
Here, we take $K_\alpha=K_\beta=0.15$. In Figure \ref{fig1}, we present the comparison of the numerical solution with the analytical solution  at $T=0.4$ with fixed $\alpha=0.4,~\beta=1.8, ~h=\tau=1/100$. In Figure \ref{fig2}, we present the behavior of the RFADE for different $\alpha$ at  $T=10.0$ with fixed $\beta=1.7, ~h=\tau=1/100$. In Figure \ref{fig3}, we present the behavior of the RFADE for different $\beta$ at  $T=10.0$ with fixed $\alpha=0.3, ~h=\tau=1/100$. In Figure \ref{fig4}, we present the behavior of the RFADE at different time $T$ with fixed $\alpha=0.4,~\beta=1.6,~h=\tau=1/100$.
\begin{figure}[H]
\centering \scalebox{0.45}[0.45]{\includegraphics{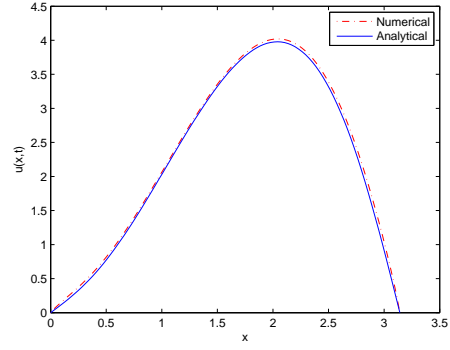}}
\caption{The comparison of the numerical solution and analytical solution  at
$T=0.4$ with fixed $\alpha=0.4,~\beta=1.8$.}
\label{fig1}
\end{figure}

\begin{figure}[H]
\centering \scalebox{0.45}[0.45]{\includegraphics{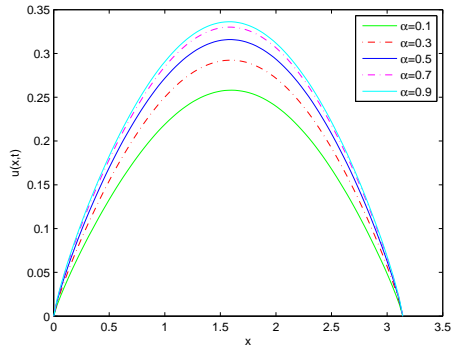}}
\caption{The numerical approximation of $u(x,t)$ for the RFADE with $\alpha=0.1,~ 0.3,~0.5,~0.7,~0.9$ at $T=10.0$ with fixed $\beta=1.7$.}
\label{fig2}
\end{figure}

\begin{figure}[H]
\centering \scalebox{0.45}[0.45]{\includegraphics{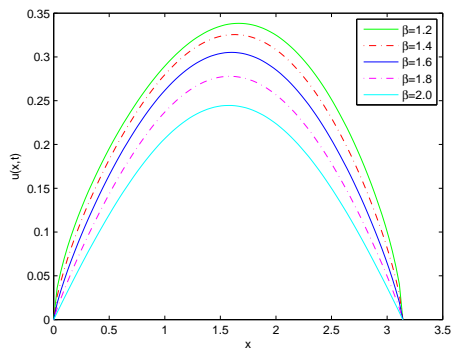}}
\caption{The numerical approximation of $u(x,t)$ for the RFADE with  $\beta=1.2,~1.4,~1.6,~1.8,~2.0$ at  $T=10.0$ with fixed $\alpha=0.3$.}
\label{fig3}
\end{figure}

\begin{figure}[H]
\centering \scalebox{0.45}[0.45]{\includegraphics{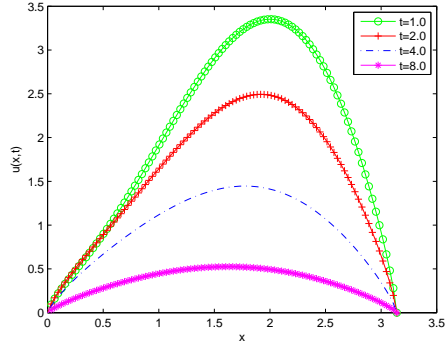}}
\caption{The numerical approximation of $u(x,t)$ for the RFADE with  $\alpha=0.4,~\beta=1.6$ at $T=1.0,~2.0,~4.0,~8.0$.}
\label{fig4}
\end{figure}

\section{Conclusions}\label{sec6}
In this paper, we have developed and demonstrated a second order {f}inite difference method for solving a class of  Riesz fractional advection-dispersion equation. Firstly, based on the WSGD operators, applying the {f}inite difference method, we derived the Crank-Nicolson scheme of the problem and rewrote the scheme as a matrix form. Subsequently, we proved that the Crank-Nicolson scheme is unconditionally stable and convergent with the accuracy of $\mathcal{O}(\tau^2+h^{2})$. Finally, some numerical results for the fractional {f}inite difference method are given to show the stability, consistency, and convergence of our computational approach. This technique can be extended to two-dimensional or three-dimensional problems with complex regions. In the future, we would like to investigate {f}inite difference method for the fractional problem in high dimensions.

\section{Acknowledgments}
This research is partially supported by the NSF of China under grant 11471274 and the Natural Science Foundation of Fujian (Grant No. 2013J01021).


\begin{thebibliography}{Lam}
\bibitem{b1}I.M. Sokolov, J. Klafter, A. Blumen, Fractional kinetics, Phys. Today Nov. (2002) 28-53.

\bibitem{b2}R.L. Magin, Fractional Calculus in Bioengineering, Begell House Publisher., Inc., Connecticut, 2006.

\bibitem{b3}S.B. Yuste, L. Acedo, K. Lindenberg, Reaction front in an $A + B\to C$ reaction-subdiffusion process, Phys. Rev. E 69 (2004) 036126.

\bibitem{b4}D.A. Benson, S.W. Wheatcraft, M.M. Meerschaert, Application of a fractional advection-dispersion equation, Water Resour. Res. 36  (2000) 1403-1412.

\bibitem{b5}D.A. Benson, S.W. Wheatcraft, M.M. Meerschaert, The fractional-order governing equation of L\'evy motion, Water Resour. Res. 36  (2000) 1413-1423.

\bibitem{b6}F. Liu, V. Anh, I. Turner, Numerical solution of the space fractional Fokker-Planck equation, J. Comput. Appl. Math. 166  (2004) 209-219.

\bibitem{b7}F. Liu, V. Anh, I. Turner, P. Zhuang, Time fractional advection-dispersion equation, J. Appl. Math. Comput. 13 (2003) 233-246.

\bibitem{b8}E. Scalas, R. Gorenflo, F. Mainardi, Fractional calculus and continuous-time finance, Phys. A: Stat. Mech. Appl. 284  (2000) 376-384.

\bibitem{b9}S. Momani, Z. Odibat, Numerical solutions of the space-time fractional advection-dispersion equation, Numer. Methods Partial Differential Equations 24  (2008) 1416-1429.

\bibitem{Kilbas06} A. A. Kilbas, H. M. Srivastava, J. J. Trujillo, Theory and applications of fractional differential equations, Elsevier, North-Holland, 2006.

\bibitem{b10}M.M. Meerschaert, C. Tadjeran, Finite difference approximations for fractional advection-dispersion flow equations, J. Comput. Appl. Math. 172 (2004) 65-77.

\bibitem{b11}P. Zhuang, F. Liu, V. Anh, I. Turner, Numerical methods for the variable-order fractional advection-diffusion equation with a nonlinear source term, SIAM J. Numer. Anal. 47 (2009) 1760-1781.

\bibitem{b12}Q. Liu, F. Liu, I. Turner, V. Anh, Approximation of the L\'evy-Feller advection-dispersion process by random walk and finite difference method, J. Comput. Phys. 222  (2007) 57-70.

\bibitem{b13}F. Liu, P. Zhuang, V. Anh, I. Turner, K. Burra, Stability and convergence of the difference methods for the space-time fractional advection-diffusion equation, Appl. Math. Comput. 191 (2007) 12-20.

\bibitem{b23}V.J. Ervin, J.P. Roop, Variational formulation for the stationary fractional advection dispersion equation, Numer. Methods Partial Differential Equations 22 (2006) 558-576.

\bibitem{b24}H. Hejazi, T. Moroney, F. Liu, Stability and convergence of a finite volume method for the space fractional advection-dispersion equation, J. Comput. Appl. Math. 255 (2014) 684-697.

\bibitem{b25}A. Golbabai, K. Sayevand, Analytical modelling of fractional advection-dispersion equation defined in a bounded space domain, Mathematical and Computer Modelling 53 (2011) 1708-1718.

\bibitem{b26}G.H. Zheng, T. Wei, Spectral regularization method for a Cauchy problem of the time fractional advection-dispersion equation, J. Comput. Appl. Math. 233 (2010) 2631-2640.

\bibitem{b27}A.R. Carella, C.A. Dorao, Least-Squares Spectral Method for the solution of a fractional advection-dispersion equation, J. Comput. Phys. 232 (2013) 33-45.

\bibitem{b21}A.I. Saichev, G.M. Zaslavsky, Fractional kinetic equations: solutions and applications, Chaos 7  (1997) 753-764.
\bibitem{b22}G.M. Zaslavsky, Chaos, fractional kinetics, and anomalous transport, Phys. Rep. 371  (2002) 461-580.

\bibitem{b18}V.V. Anh, N.N. Leonenko, Spectral analysis of fractional kinetic equations with random data, J. Stat. Phys. 104 (2001) 1349-1387.

\bibitem{b19}S. Shen, F. Liu, V. Anh, I. Turner, The fundamental solution and numerical solution of the Riesz fractional advection-dispersion equation, IMA J. Appl. Math. 73 (2008) 850-872.

\bibitem{b20}G.M. Leonenko, T.N. Phillips, On the solution of the Fokker-Planck equation using a high-order reduced basis approximation, Comput. Methods Appl. Math. 199 (2009) 158-168.

\bibitem{b14}H. Zhang, F. Liu, V. Anh, Galerkin finite element approximation of symmetric space-fractional partial differential equations, Appl. Math. Comput. 217 (2010) 2534-2545.

\bibitem{b28}Q. Yang, F. Liu, I. Turner, Numerical methods for fractional partial differential equations with Riesz space fractional derivatives, Appl. Math. Model. 34 (2010) 200-218.

\bibitem{b29}H. Ding, Y. Zhang, New numerical methods for the Riesz space fractional partial differential equations, Comput. Math. Appl. 63 (2012) 1135-1146.

\bibitem{b15}I. Podlubny, Fractional Differential Equations, Academic Press, San Diego (1999).

\bibitem{b16}W.Y. Tian, H. Zhou, W.H. Deng, A class of second order difference approximation for solving space fractional diffusion equations, Math. Comput. 84 (2015) 1703-1727.

\bibitem{b17} W. Zhang, Finite Difference Methods for Partilal Differential Equations in Science Computation, Higher Education Press, 2006 (in chinese).

\bibitem{b30}E. Isaacson, H.B. Keller, Analysis of Numerical Methods, Wiley, New York, 1966.

\end{thebibliography}
\end{document}